\documentclass[a4paper]{amsart}

\usepackage[utf8]{inputenc}
\usepackage{amsthm, amssymb, amsmath, amsfonts}
\usepackage{url}

\newtheorem{thm}{Theorem}[section]
\newtheorem{prop}[thm]{Proposition}
\newtheorem{lem}[thm]{Lemma}

\renewcommand{\le}{\leqslant}
\renewcommand{\ge}{\geqslant}

\newcommand{\mcl}{\mathcal}
\newcommand{\E}{\mathbb{E}}
\newcommand{\EE}{\mathbf{E}}
\newcommand{\EEo}{\mathbf{E}^\omega}

\newcommand{\N}{\mathbb{N}}
\renewcommand{\L}{\mathcal{L}}
\newcommand{\LL}{\mathbb{L}}
\newcommand{\1}{\mathbf{1}}
\newcommand{\R}{\mathbb{R}}

\newcommand{\Z}{\mathbb{Z}}
\renewcommand{\P}{\mathbb{P}}
\newcommand{\PP}{\mathbf{P}}
\newcommand{\PPo}{\mathbf{P}^\omega}
\newcommand{\ov}{\overline}
\newcommand{\td}{\tilde}
\newcommand{\eps}{\varepsilon}
\def\d{{\mathrm{d}}}

\newcommand{\var}{\mathbb{V}\mathrm{ar}}
\newcommand{\mfk}{\mathfrak}
\newcommand{\Ah}{A_\mathrm{hom}}

\title[A quantitative CLT for the RWRC]{A quantitative central limit theorem for the random walk among random conductances}
\author{Jean-Christophe Mourrat}

\address{Ecole polytechnique fédérale de Lausanne, institut de mathématiques, station 8, 1015 Lausanne, Switzerland}

\begin{document}
\begin{abstract}
We consider the random walk among random conductances on $\Z^d$. We assume that the conductances are independent, identically distributed and uniformly bounded away from $0$ and infinity. We obtain a quantitative version of the central limit theorem for this random walk, which takes the form of a Berry-Esseen estimate with speed $t^{-1/10}$ for $d \le 2$, and speed $t^{-1/5}$ for $d \ge 3$, up to logarithmic corrections. 


\end{abstract}
\maketitle
\section{introduction}
A classical way to represent a disordered medium is to see it as the result of a random sorting. For a conducting material, one assumes that the local conductivity $A(x)$ at point $x$ (in $\R^d$ or $\Z^d$) is a random variable. Although locally disordered, we think of the medium as having some statistical invariance in space, that is, we assume that the law of the field of conductivities is invariant under translations. 

If one is interested in a space scale that is very large compared to the typical length of the random fluctuations, then these fluctuations should average out and one should be able to replace the random medium by an equivalent \emph{homogenized} medium with a constant conductivity matrix. This problem was already considered from a physicist's point of view by Maxwell \cite{maxwell} and Rayleigh \cite{rayleigh}. It received a satisfactory mathematical treatment for periodic environments in the 70's (see for instance \cite[Chapter 1]{jko}, and references therein), and for random environments with \cite{kozlov1}, \cite{yuri1}, and \cite{papavara1}. For uniformly elliptic and ergodic environments, it was shown that there exists an effective conductivity matrix $\Ah$ such that the solution operator of $\nabla A(\cdot /\eps) \nabla$ converges, as $\eps$ tends to $0$, to the solution operator of the deterministic and homogeneous differential operator $\nabla \cdot \Ah \nabla$.

The operator $\nabla \cdot A \nabla$ defines a diffusion (or a random walk if the space is discrete) in the random medium. The probabilistic counterpart of the convergence of operators described above is the convergence of the rescaled diffusion to a Brownian motion with covariance matrix $2 \Ah$. In the discrete space setting, this central limit theorem has been proved first for the measure averaged over the environment, under increasingly general conditions on the environment by \cite{kun, kipvar, masi}. For non-elliptic i.i.d.\ environments, extending the result to convergence for almost every environment is a major recent achievement, see \cite{sid, berbisk, matpiat, bisk, mat, bardeu, andres}. For continuous space and uniformly elliptic environments, similar results were obtained by \cite{papavara1, osada}. 

Both the analytic and the probabilistic results are asymptotic. There has been some progress in turning the analytic statement into a quantitative one. \cite{yuri} and \cite{cn} prove that for uniformly elliptic environments with sufficient decorrelation, the convergence of operators is polynomial, with an exponent depending on the dimension and on the ellipticity constants. The problem of computing the homogenized matrix $\Ah$ has a similar flavor. Indeed, $\Ah$ is in general expressed as a variational problem over the full space. One must restrict it to a finite region of space for practical computations, and hence the question comes naturally to estimate the discrepancy between the true homogenized matrix and its finite volume approximation. One approach consists in computing the homogenized matrix $\Ah(n)$ associated with a periodization of the medium with periods in~$n\Z^d$. When the space is discrete, \cite{boupia} proved that $|\Ah(n) - \Ah|$ converges to $0$ polynomially fast, with an exponent that depends on the dimension $d \ge 3$ and on the ellipticity constants. Following \cite{yuri}, another approach has been analysed in \cite{go11a,go11b,gm}, that consists, instead of periodizing the medium, in introducing a $0$-order term of magnitude $1/n$ in the auxiliary problem defining the homogenized matrix. This also localizes the problem in a box of size of order $n$, and leads to other approximations of the homogenized matrix. For these approximations, explicit (and in most cases optimal) exponents of polynomial error were obtained, that depend only on the dimension.

On the probabilistic side of the problem, much less results have been obtained. As far as I know, the only exception is \cite{vardecay}, where the auxiliary process of the environment viewed by the particle is studied in discrete space, and assuming that the conductivities are bounded away from $0$. It is shown that the process converges to equilibrium polynomially fast, with an explicit exponent depending only on the dimension. An estimate on the speed of convergence to its limit of the rescaled mean square displacement of the walk is also given.

The aim of this article is to prove a quantitative central limit theorem, in the discrete space setting. We show a Berry-Esseen estimate with speed $t^{-1/10}$ for $d \le 2$, and $t^{-1/5}$ for $d \ge 3$, up to logarithmic corrections. 

%
%
%
%
%
%
%
\section{Notations and results}
\label{s:results}
\setcounter{equation}{0}
We now introduce our present setting and results with more precision. We say that $x,y \in \Z^d$ are neighbors, and write $x \sim y$, if $\|x-y\|_1 = 1$. This turns $\Z^d$ into a graph, and we write $\mathbb{B}$ for the set of (unoriented) edges thus defined. We define the \emph{random walk among random conductances} on $\Z^d$ as follows.

Let $\Omega = (0,+\infty)^{\mathbb{B}}$. An element $\omega = (\omega_e)_{e \in \mathbb{B}}$ of $\Omega$ is called an \emph{environment}. If $e = (x,y) \in \mathbb{B}$, we may write $\omega_{x,y}$ instead of $\omega_e$. By construction, $\omega$ is symmetric: $\omega_{x,y} = \omega_{y,x}$.

For any $\omega \in \Omega$, we consider the Markov process $(X_t)_{t \ge 0}$ with jump rate between $x$ and $y$ given by $\omega_{x,y}$. We write $\PPo_x$ for the law of this process starting from $x \in \Z^d$, $\EEo_x$ for its associated expectation. Its generator is given by
\begin{equation}
\label{defLom}
L^\omega f(x) = \sum_{y \sim x} \omega_{x,y} (f(y)-f(x)).
\end{equation}
The environment $\omega$ is itself a random variable, whose law we write $\P$ (and $\E$ for the corresponding expectation). We assume that 
\begin{itemize}
\item[(H1)] the random variables $(\omega_e)_{e \in \mathbb{B}}$ are independent and identically distributed,
\item[(H2)] there exists $M > 0$ such that almost surely, $\omega_e \in [1,M]$ for every $e \in \mathbb{B}$.
\end{itemize}
Naturally, imposing that $\omega_e \ge 1$ in (H2) instead of requiring the conductances to be bounded from below by a generic positive constant is simply a matter of convenience. 

Let us write $\ov{\P} = \P\PPo_0$ for the measure averaged over the environment, and $\ov{\E}$ for the associated expectation. It was shown in \cite{kipvar} that under $\ov{\P}$ and as $\eps$ tends to $0$, the process $\sqrt{\eps} X_{\eps^{-1} t}$ converges to a Brownian motion, whose covariance matrix we write $D$ (see \cite{sid} for an almost sure result under our present assumptions).

We fix once and for all some $\xi \in \R^d$, and let $\sigma > 0$ be such that $\sigma^2 = \xi \cdot D \xi$. The invariance principle ensures that
$$
\ov{\P}\left[ \xi \cdot X_t \le \sigma x \sqrt{t}  \right] \xrightarrow[t \to +\infty]{} \Phi(x),
$$
where $\Phi(x) = (2\pi)^{-1/2} \int_{-\infty}^x e^{-u^2/2} \d u$. Our aim is to get explicit bounds on the speed of convergence in the above limit. 

\begin{thm}
\label{main}
There exists $q \ge 0$ such that
$$
\sup_{x \in \R} \  \left| \ov{\P} \left[\xi \cdot X_t \le \sigma x \sqrt{t} \right] - \Phi(x) \right| =
\left|
\begin{array}{ll}
O\big(t^{-1/10}  \big) & \text{if } d = 1, \\
O\big(\log^q(t) \ t^{-1/10}  \big) & \text{if } d = 2, \\
O\big( \log(t)\ t^{-1/5} \big) & \text{if } d = 3, \\
O\big(  t^{-1/5} \big) & \text{if } d \ge 4.
\end{array}
\right.
$$
\end{thm}

\noindent \textbf{Notations}. Throughout the rest of the text, $q \ge 0$ refers to a generic constant, whose value may change from place to place and that appears only for $d = 2$. We write $\log_+(x)$ for $\max(\log(x),1)$.

%
%
%
%
%
%
%
\section{Structure of the proof}
\label{s:structure}
\setcounter{equation}{0}

Let us outline the method of proof of Theorem~\ref{main} for $d \ge 2$. 

One classical route towards an invariance principle for $(\xi \cdot X_t)_{t \ge 0}$ is to decompose the process into the sum of a martingale plus a remainder. The result can then be obtained showing that the martingale satisfies an invariance principle, and that the remainder term is negligible.

In order to prove Theorem~\ref{main}, we use this same decomposition. We will rely on a Berry-Esseen estimate for martingales due to \cite{hb70} (see also \cite{ha88}) that we now recall.

\begin{thm}[\cite{hb70}]
\label{bemart}
Let $(M(t))_{t \ge 0}$ be a locally square-integrable martingale (with respect to the probability measure $\ov{\P}$). Let $\Delta M(t) = M(t) - M(t^-)$ be its jump process, and $\langle M \rangle_t$ be its predictable quadratic variation. Define
\begin{equation}
\label{defVM}
V(M) = \ov{\E} \left[ \left( \langle M \rangle_1 - 1 \right)^2 \right],
\end{equation}
\begin{equation}
\label{defJM}
J(M) = \ov{\E}\left[ \sum_{0 \le t \le 1} (\Delta M(t))^4 \right].
\end{equation}
There exists a universal constant $C > 0$ (i.e.\ independent of $M$) such that
\begin{equation}
\label{eq:bemart}
\sup_{x \in \R} \  \left| \ov{\P} \left[M(1) \le  x \right] - \Phi(x) \right| \le C (V(M) + J(M))^{1/5}.
\end{equation}
\end{thm}

Before constructing the martingales that approximate the process $\xi \cdot X_t$, we need to introduce the following auxiliary process. 
Let $(\theta_x)_{x \in \Z^d}$ be the translations that act on the set of environments as follows: for any pair of neighbors $y,z \in \Z^d$, $(\theta_x \ \omega)_{y,z} = \omega_{x+y,x+z}$. The \emph{environment viewed by the particle} is the process defined by
\begin{equation}
\label{defenvpart}
\omega(t) = \theta_{X_t} \ \omega.
\end{equation}
One can check that $(\omega(t))_{t \ge 0}$ is a Markov process, whose generator is given by
$$
\L f (\omega) =  \sum_{|z| = 1} \omega_{0,z} (f(\theta_z \ \omega) - f(\omega)),
$$
and moreover, that the measure $\P$ is reversible and ergodic for this process. The operator $-\L$ thus defines a positive and self-adjoint operator on $\LL^2(\P)$.

Following \cite{kipvar}, let us define, for any $\mu > 0$, the function $\phi_\mu \in \LL^2(\P)$ such that
\begin{equation}
\label{defphimu}
(\mu - \L) \phi_\mu = \mathfrak{d},
\end{equation}
where the function $\mathfrak{d}$, that we call the \emph{local drift in the direction} $\xi$, is given by
\begin{equation}
\label{defmfkd}
\mathfrak{d}(\omega) = L^\omega (x \mapsto \xi \cdot x)(0) = \sum_{|z| = 1} \omega_{0,z} \ \xi \cdot z.
\end{equation}
We decompose $\xi \cdot X_t$ as the sum $M_\mu(t) + R_\mu(t)$, where
\begin{equation}
\label{defMmu}
M_\mu(t) = \xi \cdot X_t + \phi_\mu(\omega(t)) - \phi_\mu(\omega(0)) - \mu \int_0^t \phi_\mu(\omega(s)) \ \d s,
\end{equation}
and
\begin{equation}
\label{defRmu}
R_\mu(t) =  - \phi_\mu(\omega(t)) + \phi_\mu(\omega(0)) + \mu \int_0^t \phi_\mu(\omega(s)) \ \d s.
\end{equation}
\begin{prop}
\label{controlmart}
The process $(M_\mu(t))_{t \ge 0}$ is a square-integrable martingale under $\ov{\P}$ (with respect to the natural filtration associated to $(X_t)_{t \ge 0}$). Let $\sigma_\mu > 0$ be such that
\begin{equation}
\label{defsigmamu}
\sigma_\mu^2 = \sum_{|z| = 1} \E\left[\omega_{0,z} (\xi \cdot z + \phi_\mu(\theta_z \ \omega) - \phi_\mu(\omega))^2  \right].
\end{equation}
There exists $C > 0$ such that the following two inequalities hold for any $\mu, t >0$,
\begin{equation}
\label{controlV}
\ov{\E}\left[ \left( \frac{\langle M_\mu \rangle_t}{t} - \sigma_\mu^2 \right)^2 \right] \le 
\left|
\begin{array}{ll}
C \log_+^q(\mu^{-1}) \left( 1 / \sqrt{t} + \mu^2 \right) & \text{if } d = 2 , \\
C \left( \log_+(t)/t + \mu^2 \right) & \text{if } d = 3, \\
C \left( 1/t + \mu^2 \right) & \text{if } d \ge 4,
\end{array}
\right.
\end{equation}
\begin{equation}
\label{controlJ}
\frac{1}{t^2} \ov{\E}\left[ \sum_{0 \le s \le t} (\Delta M_\mu(s))^4 \right] \le 
\left|
\begin{array}{ll}
C \log_+^q(\mu^{-1}) / t & \text{if } d = 2, \\
C/t & \text{if } d \ge 3.
\end{array}
\right.
\end{equation}
\end{prop}
Proposition~\ref{controlmart} provides the estimates required to apply Theorem~\ref{bemart}. We thus obtain an explicit bound, that depends on the dimension, $\mu$, and $t$, on
\begin{equation*}
\sup_{x \in \R} \  \left| \ov{\P} \left[M_\mu(t) \le \sigma_\mu x \sqrt{t} \right] - \Phi(x) \right|.
\end{equation*}
The proof of Theorem~\ref{main} is then achieved in two steps. First, we need to control the difference between $\sigma_\mu$ and $\sigma$. Second, recalling that $\xi \cdot X_t = M_\mu(t) + R_\mu(t)$, we need to show that, for a suitable choice of $\mu$ as a function of $t$, the remainder term $R_\mu(t)$ becomes negligible in the limit. These two facts are the content of the next two propositions.
\begin{prop}
\label{p:controlsigma}
One has
\begin{equation*}
\big| \sigma_\mu^2 - \sigma^2 \big| =
\left|
\begin{array}{ll}
O\big(\mu \log^q(\mu^{-1})  \big) & \text{if } d = 2,\\
O\big(\mu^{3/2}  \big) & \text{if } d = 3,\\
O\big(\mu^2 \log(\mu^{-1})  \big) & \text{if } d = 4,\\
O\big(\mu^2  \big) & \text{if } d \ge 5.
\end{array}
\right.
\end{equation*}
\end{prop}
Proposition~\ref{p:controlsigma} is proved in \cite[Theorem~1]{go11b} (see also \cite[Theorem~3]{gm} with $k = 1$ for a slightly different point of view).
\begin{prop}
\label{controlrest}
One has
$$
\ov{\E}[(R_{1/t}(t))^2] = 
\left|
\begin{array}{ll}
O\left( \log^q(t) \right) & \text{if } d = 2, \\
O\left( 1 \right) & \text{if } d \ge 3.
\end{array}
\right.
$$
\end{prop}
We now have all the necessary information to prove Theorem~\ref{main}. 
\begin{proof}[Proof of Theorem~\ref{main} for $d \ge 2$]
Let us write 
$$
\psi(t) = 
\left|
\begin{array}{ll}
\log^q(t) \ t^{-1/10} & \text{if } d = 2, \\
\log(t)\ t^{-1/5} & \text{if } d = 3, \\
t^{-1/5}  & \text{if } d \ge 4.
\end{array}
\right.
$$
Choosing $\mu = 1/t$, we learn from Proposition~\ref{controlmart} and Theorem~\ref{bemart} that
\begin{equation}
\label{beM1t}
\sup_{x \in \R} \  \left| \ov{\P} \left[M_{1/t}(t) \le  x \sqrt{t} \right] - \Phi(x/\sigma_{1/t}) \right| = O\big(\psi(t)\big).
\end{equation}
Recalling that $\xi \cdot X_t = M_{1/t}(t) + R_{1/t}(t)$, we can write
\begin{equation}
\label{compunsens}
\ov{\P}[M_{1/t}(t) \le (x - \psi(t)) \sqrt{t}] \le \ov{\P}[\xi \cdot X_t \le x \sqrt{t}] + \ov{\P}[|R_{1/t}(t)| > \psi(t) \sqrt{t}].
\end{equation}
The second term in the right-hand side is independent of $x$ and bounded by 
$$
\frac{\ov{\E}[(R_{1/t}(t))^2]}{\psi(t)^2 t},
$$
which we know from Proposition~\ref{controlrest} to be $O(\psi(t))$. Using (\ref{beM1t}), we thus obtain that, uniformly over $x \in \R$,
\begin{equation}
\label{onesided}
\ov{\P}[\xi \cdot X_t \le x \sqrt{t}] \ge \Phi((x-\psi(t))/\sigma_{1/t})+ O\big(\psi(t)\big).
\end{equation}
Let us now show that
\begin{equation}
\label{Phiunif}	
\sup_{x \in \R} \left| \Phi((x-\psi(t))/\sigma_{1/t}) - \Phi(x/\sigma) \right|  = O\big( \psi(t) \big).
\end{equation}
In order to prove (\ref{Phiunif}), it is sufficient to consider only $x$ ranging in the interval $[-\sqrt{t},\sqrt{t}]$. For $x$ outside this interval, the bounds 
$$
\Phi(x) = O(e^{-x^2/2}) \quad (x \to - \infty) \quad \text{ and } \quad 1 - \Phi(x) = O(e^{-x^2/2}) \quad (x \to + \infty),
$$
together with the fact that $\sigma_{1/t} \to \sigma > 0$, are sufficient for the purpose of showing (\ref{Phiunif}). For $x \in [-\sqrt{t},\sqrt{t}]$, we use the fact that the derivative of $\Phi$ is bounded by $1$ to write
$$
\left| \Phi((x-\psi(t))/\sigma_{1/t}) - \Phi(x/\sigma) \right| \le |x| \left|\frac{1}{\sigma_{1/t}} - \frac{1}{\sigma}\right| + \frac{\psi(t)}{\sigma_{1/t}}.
$$
Proposition~\ref{p:controlsigma} ensures that the latter is indeed $O(\psi(t))$, uniformly over $x \in [-\sqrt{t},\sqrt{t}]$, and we have thus proved (\ref{Phiunif}).

This and inequality (\ref{onesided}) imply that, uniformly over $x \in \R$,
$$
\ov{\P}[\xi \cdot X_t \le x \sqrt{t}] \ge \Phi(x/\sigma)+ O\big(\psi(t)\big).
$$
The converse inequality is proved in the same way.
\end{proof}

\subsection*{Organization of the paper.} The rest of the paper is organized as follows. In section~\ref{s:quadvar}, we write the quadratic variation of $M_\mu$ as an additive functional of the environment viewed by the particle of the form
$$
\int_0^t v_\mu(\omega(s)) \ \d s,
$$
where $v_\mu$ is expressed in terms of the approximate corrector $\phi_\mu$. 
Section~\ref{s:decay} contains a key estimate on the decay of the variance of $v_\mu$ along the semi-group of $(\omega(s))$. Our starting point is a spatial decorrelation property of $(v_\mu(\theta_x \ \omega))_{x \in \Z^d}$ proved in \cite{go11a}, up to a minor modification that is commented on in Appendix~\ref{a:go}. We then pass to time decorrelations along the semi-group using a method from \cite{vardecay} that relies on Nash inequalities and a comparison of resolvents. The control of the fluctuations of the quadratic variation in (\ref{controlV}) is then obtained in section~\ref{s:fluctqv}. The upper bound (\ref{controlJ}) concerning the jumps of the martingale is proved in section~\ref{s:jump}. Proposition~\ref{controlrest} is then proved in section~\ref{s:rest}. Section~\ref{s:dim1} addresses the one-dimensional case. Finally, Appendix~\ref{s:append} contains some folklore facts about martingales associated to a Feller process for which I could not find a precise reference.

\subsection*{On the optimality of Theorem~\ref{main}} There seems to be no good reason for the exponents $1/10$ and $1/5$ to appear in Theorem~\ref{main}, and it is only natural to suspect that they are not optimal. On one hand, it is easy to see that one cannot hope for a better bound than $t^{-1}$ in estimates (\ref{controlV}) and (\ref{controlJ}), so the results of Proposition~\ref{controlmart} are optimal for $d \ge 3$ (provided $\mu \le t^{-1/2}$, and up to the logarithmic correction when $d = 3$). One may then wonder about the optimality of Theorem~\ref{bemart} and its not-so-intuitive exponent~$1/5$ in~(\ref{eq:bemart}). It is proved in \cite{ha88} that this exponent is optimal. However, the example provided in \cite{ha88} to show optimality is such that the maximal martingale increment is of the same order of magnitude as the martingale itself. In our context, the example is not convincing, as the martingale $M_\mu$ has ``almost bounded'' jumps (for $d \ge 3$, they are in $\LL^p(\ov{\P})$ for any $p$ uniformly over $\mu$, as can be seen using part~(ii) of Theorem~\ref{t:go}). So the question of interest to us is whether the bound $V(M)^{1/5}$ on the r.h.s.\ of~(\ref{eq:bemart}) remains optimal even on the restricted class of martingales with bounded increments. This question is answered positively in \cite{martingale_CLT}, thus leaving no possibility for improvement. On the other hand, a control of higher moments of
$$
\frac{\langle M_\mu \rangle_t}{t} - \sigma_\mu^2
$$
could allow one to use the generalized form of Theorem~\ref{bemart} given in \cite{ha88} and possibly get better exponents, but a proof that would follow this line of argument eludes me.
%
%
%
%
%
%
%
\section{The martingale $M_\mu$ and its quadratic variation}
\label{s:quadvar}
\setcounter{equation}{0}
Let us define
\begin{equation}
\label{defvmu}
v_\mu(\omega) = \sum_{|z| = 1} \omega_{0,z} (\xi \cdot z + \phi_\mu(\theta_z \ \omega) - \phi_\mu(\omega))^2 .
\end{equation}
This section is devoted to the proof of the following result.
\begin{prop}
\label{p:quadvar}
The process $M_\mu$ is a martingale under $\ov{\P}$, whose quadratic variation is given by
\begin{equation}
\label{eq:quadvar}
\langle M_\mu \rangle_t = \int_0^t v_\mu(\omega(s)) \ \d s.
\end{equation}
\end{prop}
In order to prove Proposition~\ref{p:quadvar}, we will in fact show a more general result. 
For any function $f : \Z^d \to \R$, let 
\begin{equation}
\label{defMf}
M_f(t) = f(X_t) - f(X_0) - \int_0^t L^\omega f(X_s) \ \d s,
\end{equation}
and let us define the \emph{carré du champ} of $f$ as 
$$
\Gamma_f(x) = (L^\omega f^2 - 2 f L^\omega f)(x) = \sum_{y \sim x} \omega_{x,y} (f(y) - f(x))^2.
$$
Let $B_n = \{-n,\ldots,n\}$ be the box of size $n$, and let us say that a function $f : \Z^d \to \R$ has \emph{subexponential growth} if for any $\alpha > 0$, $\sup_ {B_n} |f| = O(e^{\alpha n})$. 
\begin{prop}
\label{quadvargen}
Let $\omega$ be any environment satisfying the ellipticity condition (H2). If $f : \Z^d \to \R$ has subexponential growth, then $M_f$ defined in (\ref{defMf}) is a martingale under $\PPo_0$, whose quadratic variation is given by
$$
\langle M_f \rangle _t = \int_0^t \Gamma_f(X_s) \ \d s.
$$
\end{prop}
\begin{proof}
This statement is folklore if one assumes that $f$ is bounded, and is recalled in Appendix~\ref{s:append}. For a general $f$ of subexponential growth, let $f_n = f \1_{B_n}$. We begin by showing that $f_n(X_s)$ converges to $f(X_s)$ in $\LL^p(\PPo_0)$ for any $p > 0$, uniformly over $s \in [0,t]$. It is easy to check that, for any fixed $t \ge 0$, there exists $c > 0$ such that for any $s \le t$ and any $n$,
\begin{equation}
\label{controlexp}
\PPo_0[X_s \notin B_n] \le e^{-cn}.
\end{equation}
Indeed, this probability is bounded by the event that more that $n$ jumps occur before time $t$. As the jump rates are uniformly bounded, the number of jumps before time $t$ is dominated by a Poisson random variable, which has an exponential tail. Observe now that, for any $p > 0$,
\begin{equation}
\label{convLp}
\PPo_0\left[ \big| f_n(X_s) - f(X_s) \big|^p \right] \le \sum_{k = n}^{+\infty}  \PPo_0[X_s \in B_{k+1} \setminus B_k] \ \sup_{B_{k+1}} |f|^p.
\end{equation}
Estimate (\ref{controlexp}) and the fact that $f$ has subexponential growth together ensure that the right-hand side of (\ref{convLp}) indeed converges to $0$ as $n$ tends to infinity, uniformly over $s \in [0,t]$. 

From this observation, it is straightforward to conclude that $M_{f_n}(t)$ converges to $M_f(t)$ in $\LL^p(\PPo_0)$ for any $p$, and in particular, $M_f$ is indeed a martingale. Moreover, $\Gamma_f$ has also subexponential growth, so $\int_0^t \Gamma_{f_n}(X_s)  \d s$ converges to $\int_0^t \Gamma_{f}(X_s)  \d s$ in $\LL^p(\PPo_0)$ for any $p$, and the limit is thus the quadratic variation of $M_f$ at time $t$.
\end{proof}
\begin{proof}[Proof of Proposition~\ref{p:quadvar}]
Let $h^\omega(x) = \xi \cdot x + \phi_\mu(\theta_x \ \omega)$, and let us show that, for almost every environment, one has $M_{h^\omega} = M_\mu$ $\PPo_0$-a.s., where $M_\mu$ was defined in (\ref{defMmu}). This boils down to checking that, for almost every environment, 
\begin{equation}
\label{check1}
\forall x \in \Z^d, \quad L^\omega h^\omega(x) = \mu \phi_\mu (\theta_x \ \omega).
\end{equation}
Observe that
$$
L^\omega h^\omega(x) = \mathfrak{d}(\theta_x \ \omega) + \L \phi_\mu (\theta_x \ \omega),
$$
where $\mathfrak{d}$ is defined in (\ref{defmfkd}). We learn from the definition of $\phi_\mu$ given in (\ref{defphimu}) that, for almost every $\omega$, 
$$
\mathfrak{d}(\omega) + \L \phi_\mu(\omega) = \mu \phi_\mu(\omega).
$$
That this relation holds with probability $1$ if one replaces $\omega$ by any $\theta_x \ \omega$, $x \in \Z^d$, is a consequence of the fact that $\Z^d$ is countable, so identity (\ref{check1}) indeed holds almost surely. Moreover, as $\phi_\mu$ is integrable, the ergodic theorem ensures that $h^\omega$ has subexponential growth for almost every $\omega$, so we can apply Proposition~\ref{quadvargen}. Noting that $\Gamma_{h^\omega}(x) = v_\mu(\theta_x \ \omega)$, we thus obtain that, for almost every $\omega$, $M_\mu$ is a martingale under $\PPo_0$ whose quadratic variation is given by (\ref{eq:quadvar}). Proposition~\ref{p:quadvar} is a statement under the measure $\ov{\P}$ however. What we need in order to conclude is to check integrability, but this is straightforward due to the fact that $\phi_\mu$ is in $\LL^2(\P)$.
\end{proof}
%
%
%
%
%
%
%
\section{Polynomial decay along the semi-group}
\label{s:decay}
\setcounter{equation}{0}
As was seen in Proposition~\ref{p:quadvar}, the quadratic variation of the martingale $M_\mu$ is driven by the function $v_\mu$. In order to prove inequality (\ref{controlV}) of Proposition~\ref{controlmart}, we begin by investigating the image of $v_\mu$ by the semi-group associated with $(\omega(t))_{t \ge 0}$. Let us define
$$
v_{\mu,t}(\omega) = \EEo_0[v_\mu(\omega(t))].
$$
We are interested in the convergence to $0$ of the variance of $v_{\mu,t}$, as $t$ tends to infinity. We write $\var$ for the variance with respect to $\P$.
\begin{thm}
\label{p:decay}
There exists $C > 0$ such that for any $\mu, t  > 0$, 
\begin{equation}
\label{e:decay}
\var[v_{\mu,t}] \le 
\left|
\begin{array}{ll}
C \log_+^q(\mu^{-1})\left( 1/ \sqrt{t} + \mu^2 \right) & \text{if } d = 2, \\
C \left( \log_+(t)/t + \mu^2 \right)  & \text{if } d = 3, \\
C \left(1/t + \mu^2 \right) & \text{if } d \ge 4, 
\end{array}
\right.
\end{equation}
and moreover, 
\begin{equation}
\label{integrable}
\int_0^{t} \var[v_{\mu,s}] \ \d s \le 
\left|
\begin{array}{ll}
C \log_+^q(\mu^{-1}) \left(  \sqrt{t} + \mu^2 t \right) & \text{if } d = 2, \\
C \left(\log_+(t)  + \mu^2 t \right) & \text{if } d = 3, \\
C \left(1 +   \mu^2 t \right) & \text{if } d \ge 4. 
\end{array}
\right.
\end{equation}
\end{thm}

The idea of the proof of Theorem~\ref{p:decay} is inspired by \cite{vardecay}, with a crucial input from \cite{go11a}. Let us write $w_\mu = \mu \phi_\mu^2 + v_\mu$, and $\ov{w}_\mu = w_\mu - \E[w_\mu]$. For any function $g : \Omega \to \R$, let 
$$
S_n(g) = \sum_{x \in B_n} g(\theta_x \ \omega).
$$ 

\begin{thm}[\cite{go11a}]
\label{t:go}
\begin{itemize}
\item[(i)] There exists $C > 0$ such that, for any $n \in \N$ and any $\mu > 0$,
\begin{equation*}
\E\left[ \left(\frac{S_n(\ov{w}_\mu)}{|B_n|} \right)^2 \right] \le 
\left|
\begin{array}{ll}
C \log_+^q(\mu^{-1}) n^{-1} & \text{if } d = 2, \\
C  n^{1-d} & \text{if } d \ge 3,
\end{array}
\right.
\end{equation*}
where we write $|B_n|$ to denote the cardinal of the box $B_n$. 
\item[(ii)] For any $p > 0$, there exists $q \ge 0$ such that
$$
\E\left[\phi_\mu^p\right] =
\left|
\begin{array}{ll}
O\big( \log^q(\mu^{-1}) \big) & \text{if } d = 2,\\
O(1) & \text{if } d \ge 3.
\end{array}
\right.
$$ 
\end{itemize}
\end{thm}
Part (i) of Theorem~\ref{t:go} should inform us about the decorrelation properties of the family of random variables $(v_\mu(\theta_x \ \omega))_{x \in \Z^d}$. The proof of the estimate unfortunately requires that $v_\mu$ be replaced by $w_\mu$, which is the energy density deriving from the elliptic difference equation defining $\phi_\mu$. The result is essentially given in \cite[Theorem~2.1]{go11a}, up to a minor modification which is commented on in Appendix~\ref{a:go}. Part (ii) comes from \cite[Proposition~2.1]{go11a}. 
\begin{proof}[Proof of Theorem~\ref{p:decay}]
We need to transfer the information on the spatial decorrelations of $(w_\mu(\theta_x \ \omega))_{x \in \Z^d}$ given by part (i) of Theorem~\ref{t:go} into time decorrelations for the action of the semi-group on $w_\mu$. This is achieved using techniques from \cite{vardecay}, that are based on Nash inequalities and comparisons of resolvents. Let us define $w_{\mu,t} = \EEo_0[w_\mu(\omega(t))]$, and $\ov{w}_{\mu,t} = w_{\mu,t} - \E[w_{\mu,t}] = \EEo_0[\ov{w}_\mu(\omega(t))]$. Let $(X^\circ_t)_{t \ge 0}$ be the simple random walk (its jump rates are uniformly equal to $1$), whose distribution starting from $0$ we write $\PP_0$, and let $\ov{w}_{\mu,t}^\circ = \EE_0[\ov{w}_\mu(\theta_{X_t^\circ} \ \omega)]$. We learn from \cite[Proposition~4.1]{vardecay} that the function 
$t \mapsto \E[S_n(\ov{w}_{\mu,t}^\circ)]$ is decreasing. As a consequence, combining \cite[Proposition~7.1]{vardecay} with part (i) of Theorem~\ref{t:go}, we obtain that there exists $C > 0$ such that
\begin{equation}
\label{estimcirc}
\E[(\ov{w}_{\mu,t}^\circ)^2] \le 
\left|
\begin{array}{ll}
C \log_+^q(\mu^{-1}) \ t^{-1/2} & \text{if } d = 2, \\
C  \ t^{-(d-1)/2} & \text{if } d \ge 3.
\end{array}
\right.
\end{equation}
We then use the resolvents comparison between the simple random walk and the original one given by~\cite[Lemma~5.1]{vardecay}, that we recall here: for any $\lambda > 0$, one has
$$
\int_0^{+ \infty} e^{- \lambda s} \E[(\ov{w}_{\mu,s})^2] \ \d s \le \int_0^{+ \infty} e^{- \lambda s} \E[(\ov{w}_{\mu,s}^\circ)^2] \ \d s.
$$
This inequality holds due to the fact that we assume the conductances to be uniformly bounded from below by $1$ (see assumption (H2)). Indeed, in this case, the Dirichlet form associated to $(\omega(t))_{t \ge 0}$ dominates the Dirichlet form associated with the environment seen by the simple random walk. 

Choosing $\lambda = 1/t$ and using (\ref{estimcirc}) in the above inequality proves that
\begin{equation}
\label{integrable2}
\int_0^{t} \E[(\ov{w}_{\mu,s})^2] \ \d s \le 
\left|
\begin{array}{ll}
C \log_+^q(\mu^{-1}) \sqrt{t} & \text{if } d = 2, \\
C \log_+(t) & \text{if } d = 3, \\
C  & \text{if } d \ge 4. 
\end{array}
\right.
\end{equation}

In order to get inequality (\ref{integrable}), we observe that
$$
\var[v_{\mu,t}] = \var\left[ \left(w_\mu - \mu \phi_\mu^2  \right)_t \right],
$$
where we write $(\cdot)_t$ to denote the action of the semi-group at time $t$. This is bounded by
$$
2 \var\left[w_{\mu,t}\right] + 2 \var \left[ \left(\mu \phi_\mu^2  \right)_t \right].
$$
The first term of this sum is controlled by (\ref{integrable2}). The semi-group being a contraction in $L^2(\P)$, the second term is smaller than 
$$
\mu^2 \var\left[ \phi_\mu^2 \right] \le \mu^2 \E[\phi_\mu^4].
$$
Using part (ii) of Theorem~\ref{t:go} with $p = 4$, we bound this quantity by a constant times
$$
\left|
\begin{array}{ll}
\mu^2 \log_+^q(\mu^{-1}) & \text{if } d = 2, \\
\mu^2 & \text{otherwise},
\end{array}
\right.
$$
thus obtaining (\ref{integrable}). Claim (\ref{e:decay}) follows using the fact that the function $t \mapsto \var[v_{\mu,t}]$ is decreasing, as in the proof of \cite[Theorem~2.2]{vardecay}.
\end{proof}
%
%
%
%
%
%
%
\section{Fluctuations of the quadratic variation: a proof of (\ref{controlV})}
\label{s:fluctqv}
\setcounter{equation}{0}
\begin{proof}[Proof of estimate (\ref{controlV}) of Proposition~\ref{controlmart}]Combining the result of Proposition~\ref{p:quadvar} with the observation that $\E[v_\mu] = \sigma_\mu^2$, we have
$$
\ov{\E}\left[ \left( \frac{\langle M_\mu \rangle_t}{t} - \sigma_\mu^2 \right)^2 \right] = \frac{1}{t^2} \ \ov{\E}\left[ \left( \int_0^t \ov{v}_\mu(\omega(s)) \ \d s \right)^2 \right],
$$
where we define $\ov{v}_\mu(\omega)$ to be ${v}_\mu(\omega) - \E[{v}_\mu]$. Moreover, one has
\begin{eqnarray*}
\ov{\E}\left[ \left( \int_0^t \ov{v}_\mu (\omega(s)) \ \d s \right)^2 \right] & = & 2 \int_{0 \le s \le u \le t} \ov{\E}[ \ov{v}_\mu(\omega(s)) \ov{v}_\mu(\omega(u))] \ \d s \ \d u\\
& = & 2 \int_{0 \le s \le u \le t} \ov{\E}[ \ov{v}_\mu(\omega(0)) \ov{v}_\mu(\omega(u-s))] \ \d s \ \d u,
\end{eqnarray*}
using the stationarity of $(\omega(s))$. By a change of variables (and using the fact that $\ov{\E} = \E\EEo_0$), the latter becomes
\begin{equation*}
2 \int_0^t (t-s) \E[ \ov{v}_\mu(\omega) \ov{v}_{\mu,s}(\omega)] \ \d s,
\end{equation*}
where we write $\ov{v}_{\mu,t}(\omega) = v_{\mu,t}(\omega) - \E[v_{\mu,t}] = \EEo_0[\ov{v}_\mu(\omega(t))]$.
As the measure $\P$ is reversible for the process $(\omega(t))_{t \ge 0}$, the associated semi-group is self-adjoint in $\LL^2(\P)$, and the latter integral thus becomes
$$
2 \int_0^t (t-s) \E\left[\left(\ov{v}_{\mu,s/2}\right)^2\right] \ \d s,
$$
which can be bounded by $2 t \int_0^t  \E[(\ov{v}_{\mu,s/2})^2]  \d s$.
Estimate (\ref{controlV}) now follows from Theorem~\ref{p:decay}.
\end{proof}
%
%
%
%
%
%
%
\section{Jumps of the martingale: a proof of (\ref{controlJ})}
\label{s:jump}
\setcounter{equation}{0}
The aim of this section is to prove estimate (\ref{controlJ}) of Proposition~\ref{controlmart}, which concerns the jumps of the martingale $M_\mu$. A crucial input of the proof is a result from \cite{go11a} that we recalled as part (ii) of Theorem~\ref{t:go}. 

Let $(Y_n)_{n \in \N}$ be the sequence of sites visited by the random walk $(X_t)_{t \ge 0}$, and let $(T_n)_{n \in \N}$ be the sequence of jump instants (with $T_0 = 0$), so that
$$
X_t = Y_n \quad \text{iff} \quad T_n \le t < T_{n+1}.
$$
We can rewrite the sum that interests us using $Y_n$ and $T_n$,
$$
\sum_{0 \le s \le t} \Delta M_\mu(s)^4 = \sum_{n \in \N} \Delta M_\mu(T_{n+1})^4 \ \1_{\{ T_{n+1} \le t \}}.
$$
Let 
$$
d_\mu(\omega) = |\xi| + \sum_{|z| = 1} \big| \phi_\mu(\theta_z \ \omega) - \phi_\mu(\omega) \big|.
$$
An examination of the definition (\ref{defMmu}) of $M_\mu$ shows that 
$$
\big| \Delta M_\mu(T_{n+1}) \big| \le d_\mu(\theta_{Y_n} \ \omega),
$$
so we obtain
\begin{equation}
\label{deltadmu}
\sum_{0 \le s \le t} \Delta M_\mu(s)^4 \le \sum_{n \in \N} d_\mu^4(\theta_{Y_n} \ \omega) \ \1_{\{ T_{n+1} \le t \}}.
\end{equation}
\begin{lem}
\label{lemsumint}
There exists $C > 0$ such that for any positive function $f : \Z^d \to \R$ and any environment $\omega$ satisfying the ellipticity condition (H2),
\begin{equation}
\label{compsumint}
\EEo_0\left[ \sum_{n \in \N} f(Y_n) \ \1_{\{ T_{n+1} \le t \}} \right] \le C \EEo_0\left[ \int_0^{t+1} f(X_s) \ \d s \right].
\end{equation}
\end{lem}
\begin{proof}
We can rewrite the right-hand side of (\ref{compsumint}) as
$$
\sum_{n \in \N} \EEo_0\left[ f(Y_n) \big(T_{n+1} \wedge (t+1) - T_n \wedge (t+1) \big)\right],
$$
where $a \wedge b = \min(a,b)$. This sum is larger than
$$
\sum_{n \in \N} \EEo_0\left[ f(Y_n) \big((T_{n+1} - T_n)\wedge 1 \big) \ \1_{\{ T_{n} \le t \}} \right].
$$
Let us write $\mcl{F}_n$ for the $\sigma$-algebra generated by $Y_0,\ldots, Y_n, T_0,\ldots, T_n$. The last sum can be rewritten as
$$
\sum_{n \in \N} \EEo_0\left[ f(Y_n) \EEo_0[(T_{n+1} - T_n)\wedge 1 \ | \ \mcl{F}_n] \ \1_{\{ T_{n} \le t \}} \right].
$$
Due to the ellipticity assumption on the environment, the conditional expectation $\EEo_0[(T_{n+1} - T_n)\wedge 1 \ | \ \mcl{F}_n]$ is uniformly bounded away from $0$. We have thus proved that, for some $C > 0$,
$$
C \EEo_0\left[ \int_0^{t+1} f(X_s) \ \d s \right] \ge \sum_{n \in \N} \EEo_0\left[ f(Y_n)  \ \1_{\{ T_{n} \le t \}} \right],
$$
an inequality which implies the lemma.
\end{proof}
\begin{proof}[Proof of estimate (\ref{controlJ}) of Proposition~\ref{controlmart}]
From inequality (\ref{deltadmu}) and Lemma~\ref{lemsumint}, we get that
\begin{equation}
\label{compstat}
\ov{\E}\left[ \sum_{0 \le s \le t} (\Delta M_\mu(s))^4 \right] \le C \int_0^{t+1} \ov{\E}\left[ d_\mu^4(\omega(s)) \right] \ \d s.
\end{equation}
Due to the stationarity of the environment viewed by the particle under $\ov{\P}$, the right-hand side of (\ref{compstat}) is in fact equal to $C (t+1) \E[d_\mu^4]$. Estimate (\ref{controlJ}) of Proposition~\ref{controlmart} then follows from part (ii) of Theorem~\ref{t:go}, taking $p = 4$.
\end{proof}
%
%
%
%
%
%
%
\section{Smallness of the remainder}
\label{s:rest}
\setcounter{equation}{0}
This section is devoted to the proof of Proposition~\ref{controlrest}. It uses a spectral decomposition of the infinitesimal generator of the environment viewed by the particle. Recall that $-\L$ is a positive and self-adjoint operator on $\LL^2(\P)$. One can thus define, for any function $f \in \LL^2(\P)$, the spectral measure of $-\L$ projected on the function~$f$, that we write $e_f$ and is such that, for any bounded continuous $\Psi : [0,+\infty) \to \R$,
$$
\E\left[ f \ \Psi(-\L) f \right] = \int \Psi(\lambda) \ \d e_f(\lambda).
$$
Here is what makes this spectral representation interesting for our purpose. On one hand, one can express the $\LL^2(\ov{\P})$ norm of $R_\mu(t)$ in terms of the spectral measure associated with the local drift $\mfk{d}$. On the other hand, we have some information on the behavior of this measure close to the edge of the spectrum. This behavior is described with precision in \cite[Theorem~5]{gm} (although results given there are not optimal), but here we need only a weaker statement, that is in fact given by the case $p = 2$ of part (ii) of Theorem~\ref{t:go}. 
\begin{proof}[Proof of Proposition~\ref{controlrest}]
The random variable $R_\mu(t)$, see its definition in (\ref{defRmu}), can be decomposed as the sum of 
$$
-\phi_\mu(\omega(t))+\phi_\mu(\omega(0)) \quad \text{ and } \quad \mu \int_0^t \phi_\mu(\omega(s)) \ \d s.
$$
Recall that the process $(\omega(t))_{t \ge 0}$ is reversible under $\ov{\P}$. Applying a time reversal changes the sign of the first of the above terms, while keeping the second unchanged. As a consequence, these two are orthogonal in $\LL^2(\ov{\P})$, and thus
\begin{equation}
\label{decompsquare}
\ov{\E}\left[(R_\mu(t))^2\right] = \ov{\E}\left[\left( \phi_\mu(\omega(t)) - \phi_\mu(\omega(0)) \right)^2\right] + \mu^2 \ \ov{\E}\left[\left( \int_0^t \phi_\mu(\omega(s)) \ \d s \right)^2\right].
\end{equation}
We begin by computing the first term on the right-hand side of (\ref{decompsquare}). Expanding the square and using the fact that $\P$ is an invariant measure for $(\omega(t))$, we obtain that it is equal to
\begin{equation}
\label{avantPt}
2\E[\phi_\mu] - 2 \ov{\E}[\phi_\mu(\omega(t)) \phi_\mu(\omega)].
\end{equation}
Let us define the image of $\phi_\mu$ by the semi-group associated with $\L$, as
$$
\phi_{\mu,t}(\omega) = \EEo_0[\phi_\mu(\omega(t))] = e^{t\L} \phi_\mu  \ (\omega).
$$
Then (\ref{avantPt}) becomes 
$$
2\E[\phi_\mu] - 2 {\E}[\phi_{\mu,t} \ \phi_\mu],
$$
and using the definition (\ref{defphimu}) of $\phi_\mu$, this can be rewritten as
\begin{equation}
\label{squarep1}
2 \int \frac{1-e^{-\lambda t}}{(\lambda + \mu)^2} \ \d e_\mfk{d}(\lambda).	
\end{equation}
Let us now turn to the second term on the right-hand side of (\ref{decompsquare}). By the computation we did in section~\ref{s:fluctqv}, we readily know that
$$
\ov{\E}\left[\left( \int_0^t \phi_\mu(\omega(s)) \ \d s \right)^2\right] = 2 \int_0^t (t-s) \E[\phi_{\mu,s} \ \phi_\mu] \ \d s,
$$
which can be rewritten in terms of the spectral measure as
\begin{equation}
\label{squarep2}
2 \int \int_0^t (t-s) \frac{e^{-\lambda s}}{(\lambda + \mu)^2} \ \d s \ \d e_\mfk{d}(\lambda) = 
2 \int \frac{e^{-\lambda t} - 1 + \lambda t}{\lambda^2(\lambda + \mu)^2} \  \d e_\mfk{d}(\lambda)
\end{equation}
Combining (\ref{squarep1}) and (\ref{squarep2}), we thus obtain
$$
\ov{\E}\left[(R_\mu(t))^2\right] = 2 \int \frac{1}{(\lambda + \mu)^2} \left[1 - e^{-\lambda t} + \mu^2 \ \frac{e^{-\lambda t} - 1 + \lambda t}{\lambda^2} \right] \ \d e_\mfk{d}(\lambda).
$$
Choosing $\mu = 1/t$, one can check that the term between square brackets in the above integral remains bounded, uniformly in $\lambda$ and $t$, and thus
$$
\ov{\E}\left[(R_{1/t}(t))^2\right] \le C \int \frac{1}{(\lambda + 1/t)^2} \ \d e_\mfk{d}(\lambda).
$$
To conclude the proof, it suffices to remark that this last integral is equal to $\E[(\phi_{1/t})^2]$, and use part (ii) of Theorem~\ref{t:go}.
\end{proof}

%
%
%
%
%
%
%
\section{In dimension one}
\label{s:dim1}
\setcounter{equation}{0}
For the one-dimensional case, the easiest route is to use the function $\chi : \Z \to \R$ defined by
\begin{equation}
\label{defchi}
\chi(0) = 0 \quad \text{ and } \quad \forall x \in \Z, \ \  \chi(x+1) - \chi(x)  = \frac{\E[1/\omega_e]^{-1}}{\omega_{x,x+1}} - 1.
\end{equation}
This definition ensures that the function $x \to x + \chi(x)$ is harmonic, with $\chi(x)$ small compared to $x$. Indeed, harmonicity follows from
$$
L^\omega(x \mapsto x + \chi(x))(z)  =  \omega_{z,z+1}(1+\chi(z+1) - \chi(z)) + \omega_{z,z-1}(-1+\chi(z-1) - \chi(z)) = 0.
$$
As a consequence, we can decompose $X_t$ as $M(t) + R(t)$, where $M(t) = X_t + \chi(X_t)$ is a martingale, and $R(t) = - \chi(X_t)$ is a small remainder. As in Proposition~\ref{p:quadvar}, one can show that
$$
\langle M \rangle_t = \int_0^t v(\omega(s)) \ \d s,
$$
where 
$$
v(\omega) = \E[1/\omega_e]^{-2} \left( \frac{1}{\omega_{0,1}} + \frac{1}{\omega_{0,-1}} \right).
$$
Letting $v_t(\omega) = \EEo_0[v(\omega(t))]$, we learn from \cite[Theorem~2.2]{vardecay} that
$$
\var[v_t] = O(t^{-1/2}).
$$
As a consequence, letting $\sigma^2 = \E[v]$ and following the computations of section~\ref{s:fluctqv}, we obtain that
\begin{equation}
\label{d1qv}
\ov{\E}\left[ \left( \frac{\langle M \rangle_t}{t}  - \sigma^2 \right)^2 \right] = O(t^{-1/2}).
\end{equation}

Due to our assumption that the conductances are uniformly bounded away from~$0$, the jumps of the function $x \mapsto x + \chi(x)$ are uniformly bounded. In order to prove that 
\begin{equation}
\label{d1jump}
\ov{\E}\left[\sum_{0 \le s \le t} (\Delta M(s))^4 \right] = O(t),
\end{equation}
it thus suffices to control the number of jumps of the random walk, which can be done as in section~\ref{s:jump} (or simply by stochastically dominating this number by a Poisson process). 

Estimates (\ref{d1qv}) and (\ref{d1jump}) together imply, via Theorem~\ref{bemart}, that
$$
\sup_{x \in \R} \  \left| \ov{\P} \left[M(t) \le \sigma x \sqrt{t} \right] - \Phi(x) \right| = O(t^{-1/10}).
$$
There remains to control the rest $R(t)$. Following the argument given in the end of section \ref{s:structure}, and in particular inequality (\ref{compunsens}), what we need to check is that
$$
\ov{\P}[|R(t)| \ge \psi(t) \sqrt{t}] = O\big( \psi(t) \big),
$$
where here $\psi(t) = t^{-1/10}$, and $R(t) = - \chi(X_t)$. We need some control on the growth of the function $\chi$. As $\chi$ is the sum of bounded and centered random variables, a classical large deviation bound (or a consequence of the more refined \cite[Theorem~XVI.7.1]{fel}) yields:
\begin{lem}
\label{lemgd}
For any $\eps \in (0,1/2)$, there exists $a > 0$ such that
$$
\P[|\chi(n)| \ge n^{1/2+\eps}] \le e^{-an^{2\eps}}.
$$
\end{lem}
As the conductances are bounded away from $0$, the increments of $\chi$ are uniformly bounded by a constant $m$. Hence, on the event $|R(t)| \ge \psi(t) \sqrt{t}$, one must have $|X(t)| \ge m^{-1} \psi(t) \sqrt{t}$. As a consequence, for any $\eps \in (0,1/2)$, one has
\begin{equation}
\label{decomppsisqrt}
\ov{\P}[R(t) \ge \psi(t) \sqrt{t}] \le \P[\exists n \ge m^{-1} \psi(t) \sqrt{t} : |\chi(n)| \ge n^{1/2+\eps}] + \ov{\P}[X_t^{1/2+\eps} \ge \psi(t) \sqrt{t}].
\end{equation}
The first term on the r.h.s.\ of (\ref{decomppsisqrt}) decays faster than any negative power of $t$ due to Lemma~\ref{lemgd}. As for the second term, one can bound it by
\begin{equation}
\label{fracovEX}
\frac{\ov{\E}[X_t^2]}{\big( \psi(t) \sqrt{t} \big)^{2/(1/2+\eps)}}.
\end{equation}
The numerator of (\ref{fracovEX}) grows linearly with $t$ (see \cite[Theorem~2.1]{masi}). It thus suffices to choose $\eps$ small enough to ensure that the fraction (\ref{fracovEX}) is $O\big(\psi(t)\big)$, and this finishes the proof of Theorem~\ref{main} for $d = 1$.
%
%
%
%
%
%
\appendix
\section{On the proof of Theorem~\ref{t:go}}
\label{a:go}
\setcounter{equation}{0}
Part (i) of Theorem~\ref{t:go} is a minor variation of \cite[Theorem~2.1]{go11a}. We describe here the necessary modifications. 
What in our notation is $w_\mu(\theta_x \ \omega)$ is 
$$
T^{-1} \phi_T(x)^2 + (\nabla \phi_T(x) + \xi) \cdot A(x) (\nabla \phi_T(x) + \xi)
$$
in the notation of \cite{go11a}, with $T = 1/\mu$. Taking $n = L$ and $\eta_L = \1_{B_L}/|B_L|$, what in our notation is 
$$
\E \left[ \left( \frac{S_n(\ov{w}_\mu)}{|B_n|} \right)^2 \right]
$$
becomes in their notation
$$
\mathrm{var} \left[ \int_{\Z^d} \left( T^{-1} \phi_T(x)^2 + (\nabla \phi_T(x) + \xi) \cdot A(x) (\nabla \phi_T(x) + \xi) \right) \eta_L(x) \ \d x \right]. 
$$
\cite[Theorem~2.1]{go11a} precisely gives information about the decay of this variance, but under the assumption that the gradient of the averaging function satisfies $\| \nabla \eta_L  \|_\infty \lesssim L^{-d-1}$, while we only have $\| \nabla \eta_L  \|_\infty \lesssim L^{-d}$ here. This difference is the reason why the exponents of decay differ by $1$ between Theorem~\ref{t:go} and the original result of \cite{go11a}. 

The assumption about the gradient is used only in steps 5, 6 and 7 of the proof of \cite[Theorem~2.1]{go11a}. In step 5 (p. 810), one needs to bound
\begin{equation}
\label{compint}
\int_{\Z^d} \int_{\Z^d} |\nabla^* \eta_L(x)| |\nabla^* \eta_L(x')| \int_{\Z^d}  h(z-x) h(z-x') \ \d z \d x \d x'.
\end{equation}
For $|\nabla^* \eta_L(x)|$ to be non zero, it must be that $x \in B_{L+1} \setminus B_{L-2} =: C_L$, so up to a constant, (\ref{compint}) is bounded by
\begin{multline*}
L^{-2d} \int_{x,x' \in C_L} \int_{z \in \Z^d}  h(z-x) h(z-x') \ \d z \d x \d x' \\
= L^{-2d} \int_{x,x' \in C_L} \int_{z' \in \Z^d}  h(z') h(z'+x-x') \ \d z' \d x \d x'.
\end{multline*}
Given $x,x' \in C_L$, it is clear that $x'-x$ falls in the box of size $2L+2$. Moreover, for any $y$ in this box, there can be at most $|C_L| \sim L^{d-1}$ pairs $(x,x') \in (C_L)^2$ such that $y = x' - x$. As a consequence, (\ref{compint}) is bounded by
$$
L^{-d-1} \int_{|y| \le 2L+2} \int_{z' \in \Z^d}  h(z') h(z'-y) \ \d z' \d y.
$$
This is, up to a factor $L$, the bound that is arrived at in \cite[p.~810]{go11a}. The rest of step 5 follows without change. The very same computations apply as well in steps 6 and 7, with the same loss of a factor $L$.
%
%
%
%
%
%
%
\section{Martingales associated with a Feller process}
\label{s:append}
\setcounter{equation}{0}

Let $S$ be a Polish space, and $\mathcal{C}(S)$ be the space of all real-valued continuous functions on $S$ that tend to $0$ at infinity, equipped with the uniform norm. Let~$D$ be the space of cadlag functions from $\R_+$ to $S$, that comes together with its product $\sigma$-algebra. We write $X = (X_t)_{t \ge 0}$ for the canonical process on $D$. A \emph{Feller process} consists of a collection of probability measures $(\PP_x)_{x \in S}$ on $D$ (expectations $(\EE_x)$), together with a right-continuous and adapted filtration $(\mathcal{F}_t)_{t \ge 0}$, such that 
\begin{itemize}
\item for any $x \in S$, $\PP_x[X_0 = x] = 1$,
\item for any $f \in \mathcal{C}(S)$ and any $t \ge 0$, the mapping $x \to \EE_x[f(X_t)]$ is in $\mathcal{C}(S)$,
\item the Markov property is satisfied.
\end{itemize}
This Feller process defines a probability semi-group $(P_t)_{t \ge 0}$ on $\mathcal{C}(S)$ by $P_t f (x) = \EE_x[f(X_t)]$. This semi-group can be used to define the infinitesimal generator $L$ of the process by
\begin{equation}
\label{defL}
Lf = \lim_{t \to 0} \frac{P_t f - f}{t},
\end{equation}
for any $f$ in the set
$$
\mathcal{D}(L) = \{ f \in \mathcal{C}(S) : \text{the limit in (\ref{defL}) exists in } \mathcal{C}(S) \}.
$$
If $f$ and $f^2$ are in $\mathcal{D}(L)$, we define the \emph{carré du champ} of $f$ as $\Gamma_f = L f^2 - 2 f L f$. 
\begin{prop}
\label{p:feller}
Let $f \in \mcl{D}(L)$. The process $M_f$ defined by
$$
M_f(t) = f(X_t) - f(X_0) - \int_0^t Lf(X_s) \ \d s
$$
is a martingale under $\PP_x$, for any $x \in S$. Moreover, if $f^2 \in \mcl{D}(L)$, its predictable quadratic variation is given by
\begin{equation}
\label{quadvar}
\langle M_f \rangle_t = \int_0^t \Gamma_f(X_s) \ \d s.
\end{equation}
\end{prop}
\begin{proof}
The fact that $M_f$ is a martingale is well known, and is proved in \cite[Theorem~3.33]{lig}. The second affirmation certainly belongs to folklore, but I could not find a precise reference for it. 
Being continuous (and adapted), the process $t \mapsto \int_0^t \Gamma_f(X_s)  \d s$ is predictable. It is thus sufficient to check that the process $\td{M}$ defined by
$$
\td{M}(t) = M_f(t)^2 - \int_0^t \Gamma_f(X_s) \ \d s
$$
is a martingale. Recall that, due to our assumptions, the functions $f$, $L f$ and $L f^2$ are bounded, so there are no problems of integration. We will actually show that, for any $0\le s < t$,
\begin{equation}
\label{martlimit}
\lim_{h \to 0} h^{-1} \EE_x[\td{M}(t+h) - \td{M}(t) \ | \ \mathcal{F}_s] = 0.
\end{equation}
Let us first check that 
\begin{equation}
\lim_{h \to 0^+} h^{-1} \EE_x[\td{M}(h) - \td{M}(0)] = 0,
\end{equation}
which, using the fact that $M_f$ is itself a martingale and the right continuity of the process $X$, amounts to verify that
\begin{equation}
\label{limh}
\lim_{h \to 0^+} h^{-1} \EE_x[(M_f(h) - M_f(0))^2] = \Gamma_f(x).
\end{equation}
In order to verify (\ref{limh}), one can as well assume that $f(x) = 0$. In this case, the left hand side is equal to
$$
h^{-1} \EE_x\left[\left(f(X_h) - \int_0^h Lf(X_s) \ \d s\right)^2\right].
$$
We obtain (\ref{limh}) by developping the square and using the right continuity of the process $X$. 
Similarly, for any $h \ge 0$, we have
$$
\EE_x[\td{M}(t+h) - \td{M}(t) \ | \ \mathcal{F}_s] = \EE_x\left[(M_f(t+h) - M_f(t))^2 - \int_t^{t+h} \Gamma_f(X_u) \ \d u \ \Big| \ \mcl{F}_s\right],
$$
and the same reasoning proves the right limit of (\ref{martlimit}), including the case when $s = t$. For $s < t$, we need to check the left limit as well. The above argument can be kept unchanged provided $h \ge s-t$. We have thus shown that the function $t \mapsto \EE_x[\td{M}(t) \ | \ \mcl{F}_s]$ is differentiable and of null derivative on $(s,+\infty)$, and has null right derivative at $s$. It is thus a constant function on $[s,+ \infty)$.
\end{proof}

\end{document}